\documentclass[11pt]{amsart}
\usepackage{amsfonts}
\usepackage{amsmath,amscd}
\usepackage{amsthm}
\usepackage{amssymb}
\usepackage{latexsym}

\usepackage[usenames,dvipsnames]{color}

\usepackage[normalem]{ulem}

\usepackage{hyperref}

\newtheorem{lem}{Lemma}[section]

\newtheorem{prop}{Proposition}[section]

\newtheorem{rem}{Remark}[section]

\numberwithin{equation}{section}
\newcommand{\beqa}{\begin{eqnarray}}
\newcommand{\eeqa}{\end{eqnarray}}

\newcommand{\nc}{\newcommand}
\newcommand{\rnc}{\renewcommand}

\nc{\cal}{\mathcal}

\nc{\goth}{\mathfrak}
\rnc{\bold}{\mathbf}
\renewcommand{\frak}{\mathfrak}
\renewcommand{\Bbb}{\mathbb}

\nc\kJ{\mathfrak J}

\nc\oc{\overline{c}}
\nc\tg{\tilde{g}}
\nc\tG{\tilde{G}}
\nc\tA{\tilde{A}}
\nc\tB{\tilde{B}}
\nc\tC{\tilde{C}}
\nc\bO{\overline{O}}

\nc\ma{\mathfrak a}
\nc\mb{\mathfrak b}
\nc\mc{\mathfrak c}

\nc\K{\mathbb K}
\nc{\Cal}{\mathcal}
\nc{\Xp}[1]{X^+(#1)}
\nc{\Xm}[1]{X^-(#1)}
\nc{\on}{\operatorname}
\nc{\ch}{\mbox{ch}}
\nc{\Z}{{\bold Z}}
\nc{\J}{{\mathcal J}}
\nc{\C}{{\bold C}}
\nc{\Q}{{\bold Q}}

\nc{\N}{{\Bbb N}}
\nc\beq{\begin{equation}}
\nc\enq{\end{equation}}
\nc\lan{\langle}
\nc\ran{\rangle}
\nc\bsl{\backslash}
\nc\mto{\mapsto}
\nc\lra{\leftrightarrow}
\nc\hra{\hookrightarrow}
\nc\sm{\smallmatrix}
\nc\esm{\endsmallmatrix}
\nc\sub{\subset}
\nc\ti{\tilde}
\nc\nl{\newline}
\nc\fra{\frac}
\nc\und{\underline}
\nc\ov{\overline}
\nc\ot{\otimes}
\nc\bbq{\bar{\bq}_l}
\nc\bcc{\thickfracwithdelims[]\thickness0}
\nc\ad{\text{\rm ad}}
\nc\Ad{\text{\rm Ad}}
\nc\Hom{\text{\rm Hom}}
\nc\End{\text{\rm End}}
\nc\Ind{\text{\rm Ind}}
\nc\Res{\text{\rm Res}}
\nc\Ker{\text{\rm Ker}}
\rnc\Im{\text{Im}}
\nc\sgn{\text{\rm sgn}}
\nc\tr{\text{\rm tr}}
\nc\Tr{\text{\rm Tr}}
\nc\supp{\text{\rm supp}}
\nc\card{\text{\rm card}}
\nc\bst{{}^\bigstar\!}
\nc\he{\heartsuit}
\nc\clu{\clubsuit}
\nc\spa{\spadesuit}
\nc\di{\diamond}
\nc\cW{\cal W}
\nc\cG{\cal G}
\nc\al{\alpha}
\nc\bet{\beta}
\nc\ga{\gamma}
\nc\de{\delta}
\nc\ep{\epsilon}
\nc\io{\iota}
\nc\om{\omega}
\nc\si{\sigma}
\rnc\th{\theta}
\nc\ka{\kappa}
\nc\la{\lambda}
\nc\ze{\zeta}

\nc\vp{\varpi}
\nc\vt{\vartheta}
\nc\vr{\varrho}

\nc\Ga{\Gamma}
\nc\De{\Delta}
\nc\Om{\Omega}
\nc\Si{\Sigma}
\nc\Th{\Theta}
\nc\La{\Lambda}

\nc\op{\overline{p}}
\nc\oE{\overline{E}}
\nc\onabla{\overline{\nabla}}
\nc\oDelta{\overline{\Delta}}

\nc\bepsilon{\overline{\epsilon}}
\nc\bak{\overline{k}}
\nc\bgamma{\overline{\gamma}}
\nc\kb{\mathfrak b}
\nc\hQ{\widehat Q}
\nc\kN{\mathfrak N}
\nc\kM{\mathfrak M}
\nc\kX{\mathfrak X}
\nc\kA{\mathfrak A}

\nc\kn{\mathfrak n}
\nc\km{\mathfrak m}

\nc\tfQ{\widehat{\mathfrak Q}}

\nc\boa{\bold a}
\nc\bob{\bold b}
\nc\boc{\bold c}
\nc\bod{\bold d}
\nc\boe{\bold e}
\nc\bof{\bold f}
\nc\bog{\bold g}
\nc\boh{\bold h}
\nc\boi{\bold i}
\nc\boj{\bold j}
\nc\bok{\bold k}
\nc\bol{\bold l}
\nc\bom{\bold m}
\nc\bon{\bold n}
\nc\boo{\bold o}
\nc\bop{\bold p}
\nc\boq{\bold q}
\nc\bor{\bold r}
\nc\bos{\bold s}
\nc\bou{\bold u}
\nc\bov{\bold v}
\nc\bow{\bold w}
\nc\boz{\bold z}

\nc\ba{\bold A}
\nc\bb{\bold B}
\nc\bc{\bold C}
\nc\bd{\bold D}
\nc\be{\bold E}
\nc\bg{\bold G}
\nc\bh{\bold H}
\nc\bi{\bold I}
\nc\bj{\bold J}
\nc\bk{\bold K}
\nc\bl{\bold L}
\nc\bm{\bold M}
\nc\bn{\bold N}
\nc\bo{\bold O}
\nc\bp{\bold P}
\nc\bq{\bold Q}
\nc\br{\bold R}
\nc\bs{\bold S}
\nc\bt{\bold T}
\nc\bu{\bold U}
\nc\bv{\bold V}
\nc\bw{\bold W}
\nc\bz{\bold Z}
\nc\bx{\bold X}

\nc\ca{\mathcal A}
\nc\cb{\mathcal B}
\nc\cc{\mathcal C}
\nc\cd{\mathcal D}
\nc\ce{\mathcal E}
\nc\cf{\mathcal F}
\nc\cg{\mathcal G}
\rnc\ch{\mathcal H}
\nc\ci{\mathcal I}
\nc\cj{\mathcal J}
\nc\ck{\mathcal K}
\nc\cl{\mathcal L}
\nc\cm{\mathcal M}
\nc\cn{\mathcal N}
\nc\co{\mathcal O}
\nc\cp{\mathcal P}
\nc\cq{\mathcal Q}
\nc\car{\mathcal R}
\nc\cs{\mathcal S}
\nc\ct{\mathcal T}
\nc\cu{\mathcal U}
\nc\cv{\mathcal V}
\nc\cz{\mathcal Z}
\nc\cx{\mathcal X}
\nc\cy{\mathcal Y}

\nc\e[1]{E_{#1}}
\nc\ei[1]{E_{\delta - \alpha_{#1}}}
\nc\esi[1]{E_{s \delta - \alpha_{#1}}}
\nc\eri[1]{E_{r \delta - \alpha_{#1}}}
\nc\ed[2][]{E_{#1 \delta,#2}}
\nc\ekd[1]{E_{k \delta,#1}}
\nc\emd[1]{E_{m \delta,#1}}
\nc\erd[1]{E_{r \delta,#1}}

\nc\ef[1]{F_{#1}}
\nc\efi[1]{F_{\delta - \alpha_{#1}}}
\nc\efsi[1]{F_{s \delta - \alpha_{#1}}}
\nc\efri[1]{F_{r \delta - \alpha_{#1}}}
\nc\efd[2][]{F_{#1 \delta,#2}}
\nc\efkd[1]{F_{k \delta,#1}}
\nc\efmd[1]{F_{m \delta,#1}}
\nc\efrd[1]{F_{r \delta,#1}}

\nc\fa{\frak a}
\nc\fb{\frak b}
\nc\fc{\frak c}
\nc\fd{\frak d}
\nc\fe{\frak e}
\nc\ff{\frak f}
\nc\fg{\frak g}
\nc\fh{\frak h}
\nc\fj{\frak j}
\nc\fk{\frak k}
\nc\fl{\frak l}
\nc\fm{\frak m}
\nc\fn{\frak n}
\nc\fo{\frak o}
\nc\fp{\frak p}
\nc\fq{\frak q}
\nc\fr{\frak r}
\nc\fs{\frak s}
\nc\ft{\frak t}
\nc\fu{\frak u}
\nc\fv{\frak v}
\nc\fz{\frak z}
\nc\fx{\frak x}
\nc\fy{\frak y}

\nc\fA{\frak A}
\nc\fB{\frak B}
\nc\fC{\frak C}
\nc\fD{\frak D}
\nc\fE{\frak E}
\nc\fF{\frak F}
\nc\fG{\frak G}
\nc\fH{\frak H}
\nc\fJ{\frak J}
\nc\fK{\frak K}
\nc\fL{\frak L}
\nc\fM{\frak M}
\nc\fN{\frak N}
\nc\fO{\frak O}
\nc\fP{\frak P}
\nc\fQ{\frak Q}
\nc\fR{\frak R}
\nc\fS{\frak S}
\nc\fT{\frak T}
\nc\fU{\frak U}
\nc\fV{\frak V}
\nc\fZ{\frak Z}
\nc\fX{\frak X}
\nc\fY{\frak Y}
\nc\tfi{\ti{\Phi}}
\nc\bF{\bold F}
\rnc\bol{\bold 1}

\nc\ua{\bold U_\A}

\nc\qinti[1]{[#1]_i}
\nc\q[1]{[#1]_q}
\nc\xpm[2]{E_{#2 \delta \pm \alpha_#1}}  
\nc\xmp[2]{E_{#2 \delta \mp \alpha_#1}}
\nc\xp[2]{E_{#2 \delta + \alpha_{#1}}}
\nc\xm[2]{E_{#2 \delta - \alpha_{#1}}}
\nc\hik{\ed{k}{i}}
\nc\hjl{\ed{l}{j}}
\nc\qcoeff[3]{\left[ \begin{smallmatrix} {#1}& \\ {#2}& \end{smallmatrix}
\negthickspace \right]_{#3}}
\nc\qi{q}
\nc\qj{q}

\nc\ufdm{{_\ca\bu}_{\rm fd}^{\le 0}}

\nc\isom{\cong} 

\nc{\pone}{{\Bbb C}{\Bbb P}^1}
\nc{\pa}{\partial}

\nc{\F}{{\mathcal F}}
\nc{\Sym}{{\goth S}}
\nc{\A}{{\mathcal A}}
\nc{\arr}{\rightarrow}
\nc{\larr}{\longrightarrow}

\nc{\ri}{\rangle}
\nc{\lef}{\langle}
\nc{\W}{{\mathcal W}}
\nc{\uqatwoatone}{{U_{q,1}}(\su)}
\nc{\uqtwo}{U_q(\goth{sl}_2)}
\nc{\dij}{\delta_{ij}}
\nc{\divei}{E_{\alpha_i}^{(n)}}
\nc{\divfi}{F_{\alpha_i}^{(n)}}
\nc{\Lzero}{\Lambda_0}
\nc{\Lone}{\Lambda_1}
\nc{\ve}{\varepsilon}
\nc{\phioneminusi}{\Phi^{(1-i,i)}}
\nc{\phioneminusistar}{\Phi^{* (1-i,i)}}
\nc{\phii}{\Phi^{(i,1-i)}}
\nc{\Li}{\Lambda_i}
\nc{\Loneminusi}{\Lambda_{1-i}}
\nc{\vtimesz}{v_\ve \otimes z^m}

\nc{\asltwo}{\widehat{\goth{sl}_2}}
\nc\ag{\widehat{\goth{g}}}  
\nc\teb{\tilde E_\boc}
\nc\tebp{\tilde E_{\boc'}}

\newcommand{\eeq}{\end{equation}}
\newcommand{\ben}{\begin{eqnarray}}
\newcommand{\een}{\end{eqnarray}}

\begin{document}

\title[Askey-Wilson algebra revisited and tridiagonalization]
{Little and Big $q-$Jacobi Polynomials and the Askey-Wilson algebra}

\author{Pascal Baseilhac$^{\dagger}$, Xavier Martin$^{\dagger}$}
\address{$^\dagger$ Institut Denis-Poisson CNRS/UMR 7013 - Universit\'e de Tours - Universit\'e d'Orl\'eans
Parc de Grammont, 37200 Tours, 
FRANCE}
\email{pascal.baseilhac@idpoisson.fr; xavier.martin@idpoisson.fr}

\author{Luc Vinet$^{*}$}
\address{$^{*}$ Centre de recherches math\'ematiques Universit\'e de Montr\'eal, CNRS/UMI 3457, P.O. Box 6128, Centre-ville Station, Montr\'eal (Qu\'ebec), H3C 3J7 CANADA}
\email{vinet@CRM.UMontreal.CA}

\author{Alexei Zhedanov$^{\diamond,*}$}
\address{$^{\diamond}$ Department of Mathematics, Information School, Renmin University of China, Beijing 100872,
CHINA}
\email{zhedanov@yahoo.com}

\date{May 31, 2018}

\begin{abstract}
The little and big q-Jacobi polynomials are shown to arise as basis vectors for representations of the Askey-Wilson algebra. The operators that these polynomials respectively diagonalize are identified within the Askey-Wilson algebra generated by twisted primitive elements of $\mathfrak U_q(sl(2))$. The little q-Jacobi operator and a tridiagonalization of it are shown to realize the equitable embedding of the Askey-Wilson algebra into $\mathfrak U_q(sl(2))$.
\end{abstract}

\maketitle

\vskip -0.4cm



\vspace{7mm}


\section{Introduction}
This paper indicates how the little and big q-Jacobi polynomials \cite{KS} occur in the context of representations of the Askey-Wilson algebra. This structure was originally identified in \cite{7} as the algebra generated by the Askey-Wilson operator and the multiplication by the independent variable; its representations were shown to account for the bispectral properties of the Askey-Wilson polynomials. It has thereafter appeared in various contexts. 

The little q-Jacobi polynomials \cite{KS} have received algebraic interpretations some time ago, either as matrix elements of co-representations of the quantum group $SL_q(2)$ \cite{VS, K1, MMNNU1, MMNNU2} or equivalently \cite{FV1}, as matrix elements of q-exponentials in the generators of the quantum algebra $\mathfrak U_q(sl(2))$ \cite {FV2}. The big q-Jacobi polynomials have been connected to quantum 2-spheres \cite {NM}. We shall provide here a rather different algebraic setting for these q-polynomials as basis vectors for modules of the Askey-Wilson algebra.

On the one hand, consider twisted primitive elements in $\mathfrak U_q(sl(2))$ \cite{K2,KJ}. It is known that such elements provide an embedding of the Askey-Wilson algebra in $\mathfrak U_q(sl(2))$ \cite{GZ}. Take the holomorphic realization of $\mathfrak U_q(sl(2))$ and a specialization of the generators, it will be seen that the operators of which the little and big q-Jacobi polynomials are eigenfunctions belong to this realization of the Askey-Wilson algebra.

On the other hand, the polynomials of the Askey scheme and especially those under consideration here, are solutions of  bispectral problems. It has been shown recently how algebraic Heun operators can be associated to any such problems \cite {GrVZ}. For some choices of the generic parameters, the resulting Heun operator can be one of those corresponding to the polynomials of the Askey tableau and this provides a procedure often referred to as tridiagonalization \cite{IK1, IK2} to construct and study higher special functions from simpler ones. Clearly, bilinear combinations of the bispectral operators of classical polynomials will be tridiagonal in the eigenbasis . The main point is that such Heun operators generate when combined with one of the operators entering in the bispectral problem, an algebra that encodes the properties of the higher functions/polynomials. This algebraic viewpoint on tridiagonalization has been developed recently. The 4-parameter Wilson polynomials have thus been recovered and analysed in that perspective through the tridiagonalization of the 2-parameter hypergeometric operator \cite{GIVZ}. Similarly the characterization of the 4-parameter Bannai-Ito polynomials could be retrieved by observing that the embedding of the Bannai-Ito algebra into the superalgebra $\mathfrak osp$(1,2) involved tridiagonalizing the operator of which the little -1 Jacobi polynomials are eigenfunctions \cite {BGVZ}. It is hence natural at this point, to complete the picture by determining how the Askey-Wilson polynomials lend themselves to a tridiagonalization analysis. This is one of the goals of this paper. Typically one would work with the differential or difference operator of which the polynomials are eigenfunctions. This will also be the main approach here. It should be said that tridiagonalization can as well be applied to recurrence operators. As a matter of fact, it has already been shown that the recurrence coefficients of the Askey-Wilson polynomials can be obtained from those of the big q-Jacobi polynomials from the tridiagonalization of the recurrence operator of the latter polynomials \cite{TVZ}. 

It will be seen that the tridiagonalization of the q-difference operator, of which the little q-Jacobi polynomials are eigenfunctions, leads to the Askey-Wilson algebra in the fashion described above. It will be further observed that if the equitable presentation of $\mathfrak U_q(sl(2))$ is called upon \cite{ITW,T0}, the little  q-Jacobi operator and its tridiagonalized companion turn out to naturally take the form that generators of the Askey-Wilson algebra have when symmetrically embedded in $\mathfrak U_q(sl(2))$.

The paper will be organized as follows.  After having provided some relevant facts about  $\mathfrak U_q(sl(2))$, the Askey-Wilson algebra and its embedding in  $\mathfrak U_q(sl(2))$ are presented in Section 2. In Section 3, the little q-Jacobi polynomials are introduced via their q-difference equation, which is identified under special choices of the parameters as  the eigenvalue equation for one generator of the Askey-Wilson algebra through the realization inherited from that of  $\mathfrak U_q(sl(2))$. The eigenfunctions of a companion generator of the Askey-Wilson algebra are obtained and the expansion of the little q-Jacobi polynomials in this eigenbasis will be shown to yield an explicit expression of these polynomials. The tridiagonalization of the little q-Jacobi operator will be carried out in Section 4. This will provide an operator that is tridiagonal on the little q-Jacobi polynomials and that generates a new Askey- Wilson algebra together with the original little q-Jacobi operator. This will be obtained by noting that these operators take  the known form \cite{T}  of the Askey-Wilson generators under the embedding in  $\mathfrak U_q(sl(2))$ in terms the equitable generators of this quantum algebra \cite{T0}.
In Section 5 we shall turn to the big q-Jacobi polynomials and show that they are eigenfunctions of the same generator as for the little ones, in the initial embedding of the Askey-Wilson algebra in  $\mathfrak U_q(sl(2))$ albeit with  different and more general choices of the parameters. Their explicit expression will also be obtained as in section 3. As supplementary information, we shall show in Section 6 that a large class of big q-Jacobi polynomials can be obtained via tridiagonalization from special little q-Jacobi polynomials.
 The paper will end with a summary and outlook in Section 7.

\subsection{Notations}
In this paper, we fix a nonzero complex number $q$ which is not a root of unity.  We will use the standard $q$-shifted factorials (also called $q-$Pochhammer functions) \cite{KS}:
\beqa
(a;q)_n=\prod_{k=0}^{n-1}(1-aq^{k}).\label{poch}
\eeqa

\section{The Askey-Wilson algebra and its embedding in $\mathfrak U_q(sl(2))$}
We shall recap in this section the basic tools that shall be used in the remainder of the paper: first the standard Chevalley presentation of $\mathfrak U_q(sl(2))$  and second the equitable one, together with isomorphisms between the two. The realization of $\mathfrak U_q(sl(2))$  in terms of q-difference operators will then be recalled in subsection \ref{23} and  the Askey-Wilson algebra will be introduced via its embedding into $\mathfrak U_q(sl(2))$ in subsection \ref{24}.
\subsection{The Chevalley presentation of  $\mathfrak U_q(sl(2))$}
The Chevalley presentation of  $\mathfrak U_q(sl(2))$  consists of three generators denoted $S_\pm,s_3$. They  satisfy
\beqa
[s_3,S_\pm]=\pm S_\pm \qquad \mbox{and} \qquad
[S_+,S_-]=\frac{q^{2s_3}-q^{-2s_3}}{q-q^{-1}}\ .\label{Uqsl2}
\eeqa
The central element of  $\mathfrak U_q(sl(2))$ is the Casimir operator:
\beqa
\Omega = \frac{q^{-1}q^{2s_3}+ q q^{-2s_3}}{(q-q^{-1})^2} + S_+S_-.\label{Casimir}
\eeqa

\subsection{The equitable presentation of $\mathfrak U_q(sl(2))$}
The equitable presentation of  $\mathfrak U_q(sl(2))$  consists of three generators denoted $X,Y,Z$ \cite{ITW}.
\beqa
&& YY^{-1}=Y^{-1}Y=1,\quad 
\frac{qXY-q^{-1}YX}{q-q^{-1}}=1, \quad \frac{qYZ-q^{-1}ZY}{q-q^{-1}}=1, \quad \frac{qZX-q^{-1}XZ}{q-q^{-1}}=1.\label{equit}
\eeqa

In \cite[Lemma 5.1]{T0},  an isomorphism with the presentation (\ref{Uqsl2}) is given. Here we will use a special case. In our notation, it reads:
%
%
%
%
%
\beqa
X &=& q^{-2s_3} -(q-q^{-1})q^{1/2}S_+q^{-s_3},\label{XYZ}\\
 Y&=&q^{2s_3},\nonumber\\
Z&=& q^{-2s_3} + (1-q^{-2})q^{1/2} S_-q^{-s_3}.\nonumber
\eeqa

\subsection{The $q-$difference operators realization of $\mathfrak U_q(sl(2))$}\label{23}
An irreducible infinite dimensional representation $V_\nu$ with $2\nu \notin {\mathbb Z}_+$ can be realized by $q-$difference operators acting on the space of formal power series $f(z) = \sum_{k\in {\mathbb Z}_+} \mu_k z^k$ in the variable $z$. The lowest weight vector corresponds to $1$. Define $T_\pm(f(z))=f(q^{\pm 1} z)$. There exists an homomorphism \cite{S1}:
\beqa
&& q^{s_3}\mapsto q^{-\nu}T_+, \qquad \qquad \qquad \qquad q^{-s_3}\mapsto q^{\nu}T_-, \label{holo}\\
&& S_+ \mapsto z\frac{(q^{2\nu}T_- - q^{-2\nu}T_+)}{(q-q^{-1})},\qquad S_-\mapsto -z^{-1}\frac{(T_- - T_+)}{(q-q^{-1})}.\nonumber
\eeqa
On $V_\nu$, the eigenvalue of $\Omega$ (\ref{Casimir}) is given by:
\beqa
\omega_\nu=\frac{(q^{2\nu+1} + q^{-2\nu-1})}{(q-q^{-1})^2}.\label{Casimir-val}
\eeqa
Note that if $2\nu \in {\mathbb Z}_+$, the representation becomes reducible as $(S_+)^{2\nu+1}=0$. The corresponding invariant subspace of polynomials of degree $2\nu$. 

\subsection{The Askey-Wilson algebra}\label{24}
We now turn to the Askey-Wilson algebra and its embedding in $\mathfrak U_q(sl(2))$.
Let $c_0,\oc_0,c_1,\oc_1,\epsilon_0,\epsilon_1,\mu_0,\mu_1$ be arbitrary scalars. Define:
\beqa
W_0 &=& c_0 S_+q^{s_3} +  \oc_0 S_-q^{s_3} +  \epsilon_0 q^{2s_3} + \mu_0,\label{W}\\
W_1 &=& c_1 S_+q^{-s_3} +  \oc_1 S_-q^{-s_3} +  \epsilon_1 q^{-2s_3} + \mu_1.\nonumber
\eeqa
A third operator, namely
\beqa
G_1=qW_1W_0 -q^{-1}W_0W_1
\eeqa
 is constructed. By straightforward calculations, one finds:
\beqa
G_1=g_1 S_-^2 + g_2 S_-q^{-s_3} +  g_3 S_-q^{s_3} +   g'_2 S_+q^{-s_3} +  g'_3 S_+q^{s_3}  + g_4 q^{2s_3} + g_5q^{-2s_3} + g_6,\label{G1}
\eeqa
where $g_i,i=1,..,6$ and $g'_2,g'_3$ are expressed in terms of $c_j,\oc_j,\epsilon_j,j=0,1$:
\beqa
&&g_1=\oc_0\oc_1(q^2-q^{-2}),\nonumber\\
&& g_2=\oc_0\epsilon_1(q^2-q^{-2})q + \mu_0 \oc_1(q-q^{-1}),\quad g_3=\oc_1\epsilon_0(q^2-q^{-2})q^{-1} +  \mu_1\oc_0(q-q^{-1}),\nonumber\\
&& g'_2= \mu_0c_1(q-q^{-1}), \quad g'_3= \mu_1c_0(q-q^{-1}), \nonumber\\
&& g_4= -\frac{c_0\oc_1q^{-1}(q+q^{-1})}{(q-q^{-1})} + \mu_1\epsilon_0(q-q^{-1}),\quad g_5= -\frac{c_1\oc_0q(q+q^{-1})}{(q-q^{-1})} + \mu_0\epsilon_1(q-q^{-1}),\nonumber\\
&& g_6= (c_1\oc_0q +c_0\oc_1q^{-1})(q-q^{-1})\Omega + (\epsilon_0\epsilon_1 + \mu_0\mu_1)(q-q^{-1}).\nonumber
\eeqa

\begin{prop} $W_0,W_1,G_1$ satisfy the Askey-Wilson algebra:
\beqa 
\big[W_1,W_0\big]_q&=&G_1,\nonumber\\
\quad \big[W_0,G_1\big]_q &=& \rho_0 W_1 + \omega W_0 + \gamma_0 (W_0W_1+W_1W_0) + \gamma_1 W_0^2 + \eta_0,  \nonumber\\
\big[G_1,W_1\big]_q &=& \rho_1 W_0 + \omega W_1 + \gamma_1 (W_0W_1+W_1W_0) + \gamma_0 W_1^2 + \eta_1,  \nonumber
\eeqa
where
\beqa
\rho_0&=&-c_0\oc_0 (q+q^{-1})^2 - \mu_0^2(q-q^{-1})^2,\quad
\rho_1= -c_1\oc_1 (q+q^{-1})^2 - \mu_1^2(q-q^{-1})^2,\nonumber\\
\omega &=& (q-q^{-1})g_6 - 2\mu_0\mu_1 (q-q^{-1})^2,\nonumber\\
\gamma_0 &=& \mu_0(q-q^{-1})^2,\quad \gamma_1 = \mu_1(q-q^{-1})^2,\nonumber\\
\eta_0 &=& (q+q^{-1})\left( c_0\oc_0 \epsilon_1(q-q^{-1})^2 \Omega - \epsilon_0 (qc_1\oc_0 + q^{-1}c_0\oc_1)\right) - \mu_0(q-q^{-1})g_6 + \mu_1c_0\oc_0(q+q^{-1})^2 + \mu_0^2\mu_1(q-q^{-1})^2,\nonumber\\
\eta_1 &=& (q+q^{-1})\left( c_1\oc_1 \epsilon_0(q-q^{-1})^2 \Omega - \epsilon_1 (qc_0\oc_1 + q^{-1}c_1\oc_0)\right) - \mu_1(q-q^{-1})g_6 + \mu_0c_1\oc_1(q+q^{-1})^2 + \mu_1^2\mu_0(q-q^{-1})^2,\nonumber
\eeqa
\end{prop}
\begin{proof} Straightforward, using (\ref{Uqsl2}) and (\ref{Casimir}).\end{proof}

In the text below, we consider successively special values of the structure constants of the AW algebra. In each case, a realization of $W_0,W_1,G_1$ in terms of $\mathfrak U_q(sl(2))$ Chevalley generators is given.  For a first choice of $c_i,\oc_i,\epsilon_i,\mu_i$ namely $\mu_0=c_0=\oc_1=0$ the operator $G_1$ will be denoted by $\tG_1$. It is diagonalized by the little q-Jacobi polynomials.
For a second choice of $c_i,\oc_i,\epsilon_i,\mu_i$, namely $\mu_0=c_0=0$ the operator $G_1$ is diagonalized by the big q-Jacobi polynomials.  In both cases, the elements $W_0,W_1,G_1$ satisfy the `reduced' relations:
\beqa 
\big[W_1,W_0\big]_q&=&G_1,\nonumber\\
\quad \big[W_0,G_1\big]_q &=&  \omega W_0 + \gamma_1 W_0^2 + \eta_0,  \nonumber\\
\big[G_1,W_1\big]_q &=& \rho_1 W_0 + \omega W_1 + \gamma_1 (W_0W_1+W_1W_0)  + \eta_1,  \nonumber
\eeqa
where the structure constants are specializations of $\rho_i,\omega,\gamma_i,\eta_i$ as given above.

\section{The little $q-$Jacobi polynomials}
We shall now make precise the circumstances under which the little q-Jacobi polynomials are eigenfunctions of $\tG_1$ and form a representation basis of a specialized Askey-Wilson algebra. Given the ensuing specific choice of parameters, the eigenfunctions of $W_0$ will be determined to be q-Pochammer symbols. It will be observed that $\tG_1$ acts bidiagonally on these functions, as well as on monomials, and this will entail explicit expressions for the little q-Jacobi polynomials.
\subsection{The second order q-difference operator $\tG_1$}
If we now compare the operator $G_1$ of (\ref{G1}) using (\ref{holo}) with (3.12.4) of \cite{KS}, one has:
\beqa
\tG_1\equiv G_1|_{g'_2=g'_3=g_1=0} &\mapsto& \underbrace{\left( \frac{g_3q^{-\nu}}{(q-q^{-1})}\frac{1}{z} + g_4q^{-2\nu} \right)}_{\tA_0} T_+^2 + \underbrace{\left( -\frac{g_2q^{\nu}}{(q-q^{-1})}\frac{1}{z} + g_5q^{2\nu} \right)}_{\tB_0} T_-^2  \label{tG1exp}\\
&& + \underbrace{\left(  -\frac{(g_3q^{-\nu} - g_2q^{\nu} ) }{(q-q^{-1})}\frac{1}{z} + g_6\right)}_{\tC_0}.\nonumber
\eeqa

\begin{prop}\label{p41} For the specialization\footnote{We choose $c_0=\oc_1=\mu_0=0$ in (\ref{W}). } $g'_2=g'_3=g_1=0$, the `reduced' operator $ \tG_1$ is diagonalized by the little $q-$Jacobi polynomials. 
\end{prop}
\begin{proof} Denote the little q-Jacobi polynomials as:
\beqa
y(z) \equiv p_n(z;\ma,\mb;q^2)\nonumber
\eeqa
They satisfy the second-order q-difference equation:
\beqa
q^{-2n}(1-q^{2n})(1-\ma\mb q^{2n+2})y(z) = B(z) y(q^2z) + \overline{B}(z)y(q^{-2}z) - (B(z)+\overline{B}(z))y(z),\nonumber
\eeqa
where
\beqa
B(z) &=& -\ma\frac{1}{z}+   \ma\mb q^2 ,\nonumber\\
\overline{B}(z)&=&-\frac{1}{z}+ 1,\nonumber
\eeqa
The exact relation between the parameters $g_i$ and the parameters entering in the little q-Jacobi polynomials is as follows:
\beqa
\tG_1 y(z) =\left( g_5q^{2\nu}(q^{-2n}+\ma\mb q^{2n+2}) +g_6\right)y(z)\nonumber
\eeqa
and 
\beqa
&&\ma = -\frac{g_3}{g_5}\frac{q^{-3\nu}}{(q-q^{-1})},\quad \mb= - \frac{g_4}{g_3}(q-q^{-1})q^{-\nu-2} \quad \mbox{and}\quad \frac{g_2}{g_5}=(q-q^{-1})q^\nu.\label{lab}
\eeqa
\end{proof}

\subsection{Eigenfunctions and the explicit expression of the little $q-$Jacobi polynomials}
First, we construct the eigenfunctions of $W_0$ in (\ref{W}) for $g'_2=g'_3=g_1=0$ . 
\begin{lem}\label{lem2}\label{lem1}  For $c_0=\mu_0=0$, one has:
\beqa
W_0f_n(z) =\lambda_n f_n (z) \quad \mbox{with} \quad \lambda_n=\epsilon_0 q^{2(n-\nu)} \quad \mbox{and}\quad f_n(z) = \left(\frac{\epsilon_0}{\oc_0}(1-q^2)q^{-\nu-1}z;q^2\right)_n.\label{f}
\eeqa
\end{lem}
\begin{proof} Considering (\ref{W}) for $c_0=\mu_0=0$ and using (\ref{holo}), the action on the monomials $z^n$ reads:
\beqa
W_0 z^n= \lambda_n  z^n + \nu_n z^{n-1}\quad \mbox{with}\quad \lambda_n=\epsilon_0 q^{2(n-\nu)}\quad  \mbox{and} \quad \nu_n= \oc_0 q^{-\nu+1}\frac{1-q^{2n}}{1-q^2}.\nonumber
\eeqa
Let $f_n(z)$ be such that $W_0f_n(z)=\lambda_nf_n(z)$.  Define $f_n(z)=\sum_{s=0}^\infty \alpha_{n,s} z^s$. The action of $W_0$ on $f_n(z)$ gives the recurrence relation:
\beqa
\lambda_n \alpha_{n,s} = \lambda_s\alpha_{n,s} + \nu_{s+1}\alpha_{n,s+1}. \nonumber
\eeqa
The solution reads
\beqa
\alpha_{n,s}&=&\alpha_{n,0} \frac{(\lambda_n-\lambda_0)(\lambda_n-\lambda_1)\cdots (\lambda_n-\lambda_{s-1})}{\nu_1\nu_2 \cdots \nu_{s}}\\
&=& \alpha_{n,0}  \left(  (1-q^2)\frac{\epsilon_0}{\oc_0}q^{2n-\nu-1}\right)^s \frac{(q^{-2n};q^2)_s}{(q^2;q^2)_s},\nonumber
\eeqa
from which, setting $\alpha_{n,0}=1$  and using the $q-$binomial theorem \cite{KS}, we finally obtain (\ref{f}).
\end{proof}

On the eigenfunctions of $W_0$, we now consider the action of $\tG_1$.
\begin{lem} One has:
\beqa
\tG_1 f_n(z) = a_n f_n(z) + b_n f_{n-1}(z)\label{acttG1}
\eeqa
with
\beqa
a_n&=& g_5q^{2\nu}(q^{-2n}+\ma\mb q^{2n+2}) +g_6,\nonumber\\
b_n&=&  -g_5 q^{2\nu}q^{-2n}(1-q^{2n})(1-q^{2n}\mb).\nonumber
\eeqa
\end{lem}
\begin{proof} Recall (\ref{tG1exp}). The l.h.s. of (\ref{acttG1}) reads:
\beqa
\tG_1 f_n(z) &=& \left(\tA_0 (1-\gamma_0 q^{2n-2}z)(1-\gamma_0 q^{2n}z) + \tB_0 (1-\gamma_0 q^{-2}z)(1-\gamma_0 z) \right.\nonumber\\
&& \left.+ \tC_0 (1-\gamma_0 z)(1-\gamma_0 q^{2n-2}z)\right)(1-\gamma_0 q^{2}z)\cdots (1-\gamma_0 q^{2n-4}z),\nonumber
\eeqa
where we have denoted  $\gamma_0=\frac{\epsilon_0}{\oc_0}(1-q^2)q^{-\nu-1}$. The r.h.s reads:
\beqa
\left(a_n (1-\gamma_0 z)(1-\gamma_0 q^{2n-2}z) + b_n(1-\gamma_0 z)\right)(1-\gamma_0 q^{2}z)\cdots (1-\gamma_0 q^{2n-4}z).\nonumber
\eeqa 
Equating both sides of the equation (\ref{actG1}), the coefficients $a_n$ and $b_n$ are determined uniquely. 
\end{proof}
\begin{rem} Observe that the coefficient $a_n$ coincides with the spectrum of $\tG_1$, as expected ($\tG_1$ is lower triangular in the basis $\{f_n\}$).
\end{rem}

Next, we are interested in the overlap coefficients between the eigenfunctions of $W_0$ (the $f_n(z)$) and the eigenfunctions of $\tG_1$ (the little q-Jacobi polynomials).
\begin{prop}\label{prop2} One has:
\beqa
p_n(z;\ma,\mb;q^2)= \sum_{s=0}^{\infty} \gamma_{n,s}f_s(z) \label{exppn}
\eeqa
where
\beqa
\gamma_{n,s}=   (-q^2 \mb)^{-n} q^{-n(n-1)} \frac{(q^2\mb;q^2)_n}{(q^2\ma;q^2)_n} \frac{ (q^{-2n};q^2)_s (\ma\mb q^{2n+2};q^2)_s}{(q^2\mb;q^2)_s (q^2;q^2)_s}q^{2s} .\label{gam}
\eeqa
\end{prop}
\begin{proof} By Prop. \ref{p41} recall that $p_n(z;\ma,\mb;q^2)$ are eigenfunctions of $\tG_1$ with the identification (\ref{lab}).  Consider the expansion (\ref{exppn}).
The action of $\tG_1$ on $p_n(z;\ma,\mb;q^2)$ gives the recurrence relation:
\beqa
a_n \gamma_{n,s} = a_s\gamma_{n,s} + b_{s+1}\gamma_{n,s+1}. \nonumber
\eeqa
The solution reads
\beqa
\gamma_{n,s}&=&\gamma_{n,0} \frac{(a_n-a_0)(a_n-a_1)\cdots (a_n-a_{s-1})}{b_1b_2 \cdots b_{s}}.\nonumber
\eeqa
By straightforward calculations, one finds:
\beqa
(a_n-a_0)(a_n-a_1)\cdots (a_n-a_{s-1}) &=& (-g_5q^{2\nu})^s q^{-s(s-1)} (q^{-2n};q^2)_s (\ma\mb q^{2n+2};q^2)_s,\nonumber\\
b_1b_2\cdots b_s &=&  (-g_5q^{2\nu})^sq^{-s(s+1)}(q^2\mb;q^2)_s (q^2;q^2)_s.\nonumber
\eeqa
Setting the normalization such that:
\beqa
 \gamma_{n,0}= (-q^2 \mb)^{-n} q^{-n(n-1)} \frac{(q^2\mb;q^2)_n}{(q^2\ma;q^2)_n},
\eeqa
we get (\ref{gam}).
\end{proof}

\begin{rem} The expansion formula (\ref{exppn}) coincides with (2.46) of \cite{Kor} with $q \rightarrow q^2$:
\beqa
p_n(z;\ma,\mb;q^2) =  (-q^2 \mb)^{-n} q^{-n(n-1)} \frac{(q^2\mb;q^2)_n}{(q^2\ma;q^2)_n} \    _3\phi_2\left[\begin{matrix} q^{-n} , q^{n+1}\ma\mb, q^2\mb z \\
q^2\mb, 0 \end{matrix}\,; q^2,q^2\right].
\eeqa
\end{rem}
\begin{rem} Note that the little q-Jabobi operator $\tG_1$ is also bidiagonal in the monomial basis $z^n$.  By analogy with the proof of Lemma \ref{lem2} one can retrieve the following familar explicit expression for the little q-Jacobi polynomial \cite[eq. (3.12.1)]{KS} (see also \cite[eq. (2.45)]{Kor}):
\beqa
p_n(z;\ma,\mb;q^2) = \   _2\phi_1(q^{-n} , q^{n+1}\ma\mb, q^2\ma ; q^2;q^2 z).
\eeqa
\end{rem}

\vspace{2mm}

\section{Tridiagonalization of the little $q-$Jacobi operator and the equitable presentation of $\mathfrak U_q(sl(2))$}
Calling upon the equitable embedding of the Askey-Wilson algebra \cite{T} which will be recalled next, we shall now make the observation that the little q-Jacobi operator and a tridiagonalization of this operator realizes also the Askey-Wilson algebra.
\begin{rem} Define
\beqa
g_2=(1-q^{-2})q^{1/2}b^{-1},\quad g_3=-q^{-3/2}(q-q^{-1})ca^{-1},\quad g_4=g_5^{-1}=b, \quad g_6=0.
\eeqa
Then
\beqa
\tG_1&=& (1-q^{-2})q^{1/2}b^{-1}  S_-q^{-s_3} -  q^{-3/2}(q-q^{-1})ca^{-1}  S_-q^{s_3} + b q^{2s_3} + b^{-1}q^{-2s_3}\label{B}\\
 &=& bY + b^{-1}Z +qca^{-1}(1-YZ)\ ,\nonumber
\eeqa
where $Y$ and $Z$ are the equitable $\mathfrak U_q(sl(2))$ generators introduced in (\ref{XYZ}).

\end{rem}

Using the equitable presentation of $\mathfrak U_q(sl(2))$  (\ref{XYZ}), by Proposition 1.1 and Lemma 3.4 in \cite{T}, one gets:
\begin{prop} Define
\beqa
A&=& aX + a^{-1}Y + qbc^{-1}(1-XY),\label{ABC}\\
B&=&  bY + b^{-1}Z +qca^{-1}(1-YZ),\nonumber\\
C&=& cZ + c^{-1}X + qab^{-1}(1-ZX).\nonumber
\eeqa
The generators $A,B,C$ satisfy the Askey-Wilson algebra:
\beqa
A + \frac{qBC-q^{-1}CB}{q^2-q^{-2}} &=& \frac{\Lambda(a+a^{-1}) + (b+b^{-1})(c+c^{-1})}{q+q^{-1}}\ ,\nonumber\\
B + \frac{qCA-q^{-1}AC}{q^2-q^{-2}} &=& \frac{\Lambda(b+b^{-1}) + (c+c^{-1})(a+a^{-1})}{q+q^{-1}}\ ,\nonumber\\
C + \frac{qAB-q^{-1}BA}{q^2-q^{-2}} &=& \frac{\Lambda(c+c^{-1}) + (a+a^{-1})(b+b^{-1})}{q+q^{-1}}\ ,\nonumber
\eeqa
where $\Lambda$ denotes the `normalized' Casimir element
\beqa
\Lambda=(q-q^{-1})^2\Omega.\nonumber
\eeqa
\end{prop}
\begin{rem} In terms of the Chevalley generators of $\mathfrak U_q(sl(2))$, one has $B=\tG_1$ with (\ref{B}) and
\beqa
A&=&  -a(q-q^{-1})q^{1/2}S_+q^{-s_3} + bc^{-1}(q^2-1)q^{1/2} S_+q^{s_3} +   a^{-1}q^{2s_3} +  aq^{-2s_3} ,\label{A}\\
C&=& (c+c^{-1})q^{-2s_3}  - ab^{-1}(q+q^{-1})q^{-4s_3} + ab^{-1}(q-q^{-1})^2\Omega q^{-2s_3} \nonumber\\
&& + (1-q^{-2}) \left(S_-(cq^{1/2} q^{-s_3} -q^{3/2}ab^{-1}q^{-3s_3})  - S_+(c^{-1}q^{3/2}q^{-s_3} -ab^{-1}q^{1/2}q^{-3s_3})\right).\nonumber
\eeqa
\end{rem}

For the analysis to follow, let us introduce the operators $O$ and $\bO$:
 \beqa
&& O = (q-q^{-1}) q^{-\nu} S_+ q^{s_3}(1-q^{-2\nu}q^{2s_3})^{-1}, \quad \bO = (q-q^{-1}) q^{-\nu-1} q^{-s_3}S_-(1-q^{-2\nu}q^{-2s_3})^{-1},\label{ObO}
\eeqa  
where $(1-q^{-2\nu}q^{\pm 2s_3})^{-1}$ is understood as a power series in the Cartan element of $\mathfrak U_q(sl(2))$. Note that
$O\bO=\bO O=1$. Also,  the action of each operator  on the space of monomials in $z$, according to (\ref{holo}), is such that:
\beqa
O \mapsto z, \quad \bO \mapsto z^{-1}.
\eeqa

\begin{lem}\label{ABz} Consider $A,B$ as the $q-$difference operators obtained from (\ref{ABC}) with (\ref{XYZ}). One has:
\beqa
A \equiv \alpha O B + \beta B O + \gamma O  + \delta ,\label{Texp}
\eeqa
with
\beqa
\alpha&=& -\frac{q(q^{3/2+\nu}ab-q^{-3/2-\nu}c^{-1})}{q^2-q^{-2}} ,\quad \beta= \frac{(q^{1/2+\nu}ab-q^{3/2-\nu}c^{-1})}{q^2-q^{-2}} ,\label{sol1}\\
 \gamma &=&  bc^{-1}q^{3/2+\nu} +aq^{1/2-\nu},\nonumber\\
 \delta&=&  \frac{(q^{2\nu+1}a+ q^{-2\nu-1}a^{-1}  +bc + b^{-1}c^{-1})}{q+q^{-1}}. \nonumber
\eeqa
\end{lem}
\begin{proof}
Recall that $B=\tG_1$ is given by (\ref{B}). Consider the r.h.s. of (\ref{Texp}). The action of $O$ on the monomials $S_-q^{\pm s_3},q^{\pm 2s_3}$ is such that:
\beqa
OS_-q^{-s_3} &=& \frac{q^{\nu} -q^{-\nu}q^{-2s_3}}{q-q^{-1}}, \quad  OS_-q^{s_3} = \frac{-q^{-\nu} +q^{\nu}q^{2s_3}}{q-q^{-1}},\nonumber\\
S_-q^{-s_3}O &=& \frac{q^{\nu} -q^{-\nu-2}q^{-2s_3}}{q-q^{-1}}, \quad  S_-q^{s_3}O = \frac{-q^{-\nu} +q^{\nu+2}q^{2s_3}}{q-q^{-1}},\nonumber\\
q^{2s_3}O &=& -(q-q^{-1})q^{\nu+2}S_+q^{s_3} +q^{2\nu+2}O,\quad q^{-2s_3}O = (q-q^{-1})q^{-\nu-2}S_+q^{-s_3} +q^{-2\nu-2}O.\nonumber
\eeqa
 By straightforward calculations, we obtain:
\beqa
\alpha OB + \beta BO + \gamma O  + \delta &=& -b(q-q^{-1})q^\nu(\alpha+\beta q^2)S_+q^{s_3} + b^{-1}(q-q^{-1})q^{-\nu}(\alpha+\beta q^{-2})S_+q^{-s_3} \nonumber\\
&& -  ca^{-1}q^{-3/2}q^\nu(\alpha+\beta q^2)q^{2s_3} -   b^{-1}q^{-1/2}q^{-\nu}(\alpha+\beta q^{-2})q^{-2s_3} \nonumber\\
&&   + \left( \alpha(bq^{2\nu} + b^{-1}q^{-2\nu}) + \beta (bq^{2\nu+2} + b^{-1}q^{-2\nu-2}) + \gamma \right)O\nonumber\\
&&  + \left(b^{-1}q^{\nu-1/2} + ca^{-1}q^{-\nu-3/2}\right)(\alpha+\beta) + \delta.\nonumber
\eeqa
Compare the expression above with (\ref{A}), we obtain the constraints:
\beqa
&& a=-b^{-1}q^{-1/2-\nu}(\alpha+\beta q^{-2}), \quad c = -\frac{q^{-\nu+3/2}}{(\alpha+\beta q^{2})}, \nonumber\\
&&  \alpha(bq^{2\nu} + b^{-1}q^{-2\nu}) + \beta (bq^{2\nu+2} + b^{-1}q^{-2\nu-2}) + \gamma = 0,\nonumber\\
&& \left(b^{-1}q^{\nu-1/2} + ca^{-1}q^{-\nu-3/2}\right)(\alpha+\beta) + \delta =0\nonumber
\eeqa
from which we get (\ref{sol1}).
\end{proof}

This lemma thus states that the generator $A$ of the equitable presentation of the Askey-Wilson algebra is obtained from $B$ by tridiagonalization. Stated differently, the upshot is that $B$ i.e. $\tG_1$ together with its tridiagonalization (\ref{Texp})
generate the Askey-Wilson algebra.

The tridiagonalized form of $A$ can be inverted to express $B$ in terms of $A$ as follows.
\begin{lem}  Consider $A,B$ as the $q-$difference operators obtained from (\ref{ABC}) with (\ref{XYZ}). One has:
\beqa
B \equiv \alpha \bO A + \beta A \bO + \gamma \bO  + \delta ,\label{TBexp}
\eeqa
with
\beqa
\alpha&=& -\frac{q^{-1}(q^{3/2+\nu}c-q^{-3/2-\nu}a^{-1}b^{-1})}{q^2-q^{-2}} ,\quad \beta= \frac{(q^{-3/2+\nu}c-q^{-1/2-\nu}a^{-1}b^{-1})}{q^2-q^{-2}} ,\label{sol2}\\
 \gamma &=&  b^{-1}q^{-1/2+\nu} +ca^{-1}q^{-3/2-\nu},\nonumber\\
 \delta&=&  \frac{(q^{2\nu+1}b+ q^{-2\nu-1}b^{-1}  + ac + a^{-1}c^{-1})}{q+q^{-1}}.\nonumber
\eeqa
\end{lem}
\begin{proof}
Recall $A$ is given by (\ref{A}). Consider the r.h.s. of (\ref{TBexp}). The action of $\bO$ on the monomials $S_+q^{\pm s_3},q^{\pm 2s_3}$ is such that:
\beqa
\bO S_+q^{s_3} &=& \frac{q^{\nu} -q^{-\nu}q^{2s_3}}{q-q^{-1}}, \quad  \bO S_+q^{-s_3} = \frac{-q^{-\nu} +q^{\nu}q^{-2s_3}}{q-q^{-1}},\nonumber\\
S_+q^{s_3}\bO &=& \frac{q^{\nu} -q^{-\nu-2}q^{2s_3}}{q-q^{-1}}, \quad  S_+q^{-s_3}\bO = \frac{-q^{-\nu} +q^{\nu+2}q^{-2s_3}}{q-q^{-1}},\nonumber\\
q^{-2s_3}\bO &=& -(q-q^{-1})q^{\nu+2}S_-q^{-s_3} +q^{2\nu+2}\bO,\quad q^{2s_3}\bO = (q-q^{-1})q^{-\nu-2}S_-q^{s_3} +q^{-2\nu-2}\bO.\nonumber
\eeqa
 By straightforward calculations, we obtain:
\beqa
\alpha \bO A + \beta A\bO + \gamma \bO  + \delta &=& a^{-1}(q-q^{-1})q^{-\nu}(\alpha+\beta q^{-2})S_-q^{s_3} - a(q-q^{-1})q^{\nu}(\alpha+\beta q^{2})S_-q^{-s_3} \nonumber\\
&& -  bc^{-1}q^{3/2-\nu}(\alpha+\beta q^{-2})q^{2s_3} -   aq^{\nu+1/2}(\alpha+\beta q^{2})q^{-2s_3} \nonumber\\
&&   + \left( \alpha(aq^{2\nu} + a^{-1}q^{-2\nu}) + \beta (aq^{2\nu+2} + a^{-1}q^{-2\nu-2}) + \gamma \right)\bO\nonumber\\
&&  + \left(aq^{1/2-\nu} + bc^{-1}q^{3/2+\nu}\right)(\alpha+\beta) + \delta.\nonumber
\eeqa
Compare the expression above with (\ref{B}), we obtain the constraints:
\beqa
&& c= -(\alpha + q^{-2}\beta)q^{3/2-\nu},\quad b^{-1}= - (\alpha + q^2\beta)aq^{\nu+1/2}, \nonumber\\
&&  \alpha(aq^{2\nu} + a^{-1}q^{-2\nu}) + \beta (aq^{2\nu+2} + a^{-1}q^{-2\nu-2}) + \gamma = 0,\nonumber\\
&& \left(aq^{1/2-\nu} + bc^{-1}q^{3/2+\nu}\right)(\alpha+\beta) + \delta =0\nonumber
\eeqa
from which we get (\ref{sol2}).
\end{proof}

\section{The big $q-$Jacobi polynomials}
We shall now carry for the big q-Jacobi polynomials an analysis similar to the one that was given in Section 3 for the little q-Jacobi polynomials; namely, identify the specialization of the parameters in $G_1$ that will lead to the big q-Jacobi operator, identify the eigenfunctions of $W_0$ in that case, determine the action of the restricted $G_1$ on these functions and arrive at the explicit expression of the big q-Jacobi polynomials by expanding one eigenbasis over the other.
\subsection{The second order q-difference operator $G_1$}
Recall (\ref{holo}). By straightforward replacements, from (\ref{G1}) we get for the specialization\footnote{We choose $\mu_0=c_0=0$ in (\ref{W}).} $g'_2=g'_3=0$:
\beqa
G_1|_{g'_2=g'_3=0} &\mapsto& \underbrace{ \left( \frac{g_1q^{-1}}{(q-q^{-1})^2}\frac{1}{z^2}  +\frac{g_3q^{-\nu}}{(q-q^{-1})}\frac{1}{z} + g_4q^{-2\nu} \right) }_{A_0}T_+^2 \label{G1exp}\\
&& + \underbrace{\left( \frac{g_1q}{(q-q^{-1})^2}\frac{1}{z^2}  -\frac{g_2q^{\nu}}{(q-q^{-1})}\frac{1}{z} + g_5q^{2\nu} \right) }_{B_0}T_-^2  \nonumber\\
&& + \underbrace{\left(  -\frac{g_1(q+q^{-1})}{(q-q^{-1})^2}\frac{1}{z^2}  -\frac{(g_3q^{-\nu} - g_2q^{\nu} ) }{(q-q^{-1})}\frac{1}{z} + g_6\right)}_{C_0}.\nonumber
\eeqa

\begin{prop} For  $g'_2=g'_3=0$, the operator $G_1$ is diagonalized by the big q-Jacobi polynomials.
\end{prop}
\begin{proof} Compare the operator $G_1$ written as (\ref{G1exp}) with the second-order q-difference operator\footnote{Here $q^2$ is $q$ of \cite{KS}.} (3.5.4) in \cite{KS}. Denote the big q-Jacobi polynomials as:
\beqa
y(z) \equiv P_n(z;\ma,\mb,\mc;q^2)\nonumber
\eeqa
They satisfy the second-order q-difference equation:
\beqa
q^{-2n}(1-q^{2n})(1-\ma\mb q^{2n+2})y(z) = B(z) y(q^2z) + \overline{B}(z)y(q^{-2}z) - (B(z)+\overline{B}(z))y(z),\label{qdiff}
\eeqa
where
\beqa
B(z) &=& \ma\mc q^2 \frac{1}{z^2} - \ma(\mb+\mc)q^2 \frac{1}{z} + \ma\mb q^2,\nonumber\\
\overline{B}(z)&=&\ma\mc q^4 \frac{1}{z^2} - (\ma+\mc)q^2 \frac{1}{z} + 1,\nonumber
\eeqa
The exact relation between the parameters $g_i$ and the parameters entering in the big q-Jacobi polynomials is as follows:
\beqa
G_1 y(z) =\left( g_5q^{2\nu}(q^{-2n}+\ma\mb q^{2n+2}) +g_6\right)y(z)\nonumber
\eeqa
and 
\beqa
&&\ma\mc = \frac{g_1}{g_5}\frac{q^{-2\nu-3}}{(q-q^{-1})^2},\quad \ma(\mb+\mc)= -\frac{g_3}{g_5}\frac{q^{-3\nu-2}}{(q-q^{-1})},\quad  \ma\mb = \frac{g_4}{g_5}q^{-4\nu-2},\quad \ma+\mc =\frac{g_2}{g_5}\frac{q^{-\nu-2}}{(q-q^{-1})}.
\eeqa
\vspace{2mm}

\begin{rem}Note that if instead we would like to consider the spectral problem for the big q-Jacobi polynomials 
\beqa
y(q^2\mb z) \equiv P_n(q^2\mb z;\ma,\mb,\mc;q^2),
\eeqa
the substitution $z \rightarrow q^2\mb z$ into (\ref{qdiff}) gives the following identification:
\beqa
&&\frac{\ma\mc}{\mb^2} = \frac{g_1}{g_5}\frac{q^{-2\nu+1}}{(q-q^{-1})^2},\quad \frac{\ma(\mb+\mc)}{\mb}= -\frac{g_3}{g_5}\frac{q^{-3\nu}}{(q-q^{-1})},\quad  \ma\mb = \frac{g_4}{g_5}q^{-4\nu-2},\quad \frac{\ma+\mc}{\mb} =\frac{g_2}{g_5}\frac{q^{-\nu}}{(q-q^{-1})}.
\eeqa
For the special choice $\mc=0$, note that the big q-Jacobi polynomial $P_n(q^2\mb z;\mb,\ma,0;q^2)$ can be expressed in terms of the little q-Jacobi polynomial $p_n( z;\ma,\mb;q^2)$. See \cite{KS} for details.
\end{rem}

\end{proof}

\subsection{Eigenfunctions and the explicit expression of the big $q-$Jacobi polynomials}
Consider the eigenfunctions of $W_0$ in (\ref{W}) for $g'_2=g'_3=0$. They coincide with the ones given in Lemma \ref{lem1}.
On these eigenfunctions of $W_0$, we now consider the action of $G_1$.
\begin{lem} One has:
\beqa
G_1 f_n(z) = a_n f_n(z) + b_n f_{n-1}(z)\label{actG1}
\eeqa
with
\beqa
a_n&=& g_5q^{2\nu}(q^{-2n}+\ma\mb q^{2n+2}) +g_6,\nonumber\\
b_n&=&  - g_5 q^{2\nu}q^{-2n}(1-q^{2n})(1-q^{2n}\ma)(1-q^{2n}\mc).\nonumber
\eeqa
\end{lem}
\begin{proof} Recall (\ref{G1exp}). The l.h.s. of (\ref{actG1}) reads:
\beqa
G_1 f_n(z) &=& \left(A_0 (1-\gamma_0 q^{2n-2}z)(1-\gamma_0 q^{2n}z) + B_0 (1-\gamma_0 q^{-2}z)(1-\gamma_0 z) \right.\nonumber\\
&& \left.+ C_0 (1-\gamma_0 z)(1-\gamma_0 q^{2n-2}z)\right)(1-\gamma_0 q^{2}z)\cdots (1-\gamma_0 q^{2n-4}z),\nonumber
\eeqa
where  $\gamma_0=\frac{\epsilon_0}{\oc_0}(1-q^2)q^{-\nu-1}$. The r.h.s reads:
\beqa
\left(a_n (1-\gamma_0 z)(1-\gamma_0 q^{2n-2}z) + b_n(1-\gamma_0 z)\right)(1-\gamma_0 q^{2}z)\cdots (1-\gamma_0 q^{2n-4}z).\nonumber
\eeqa 
Equating both sides of the equation (\ref{acttG1}), the coefficients $a_n$ and $b_n$ are determined uniquely. 
\end{proof}
\begin{rem} Observe that the coefficient $a_n$ coincides with the spectrum of $G_1$, as expected ($G_1$ is lower triangular in the basis $\{f_n\}$).
\end{rem}

Next, we are interested in the overlap coefficients between the eigenfunctions of $W_0$ (the $f_n(z)$) and the eigenfunctions of $G_1$ (the big q-Jacobi polynomials). The proof of the following proposition is straightforward, by analogy with the proof of Prop. \ref{prop2}
\begin{prop} One has:
\beqa
P_n(q^2 \mb z;\ma,\mb,\mc;q^2)= \sum_{s=0}^{\infty} \gamma'_{n,s}f_s(z) \label{expPn}
\eeqa
where
\beqa
\gamma'_{n,s}=  \frac{ (q^{-2n};q^2)_s (\ma\mb q^{2n+2};q^2)_s}{(q^2\ma;q^2)_s (q^2\mc;q^2)_s (q^2;q^2)_s}q^{2s} .\label{gam2}
\eeqa
\end{prop}
\begin{rem} In particular,  note that:
\beqa
 p_n(z;\ma,\mb;q^2) =  (-q^2 \mb)^{-n} q^{-n(n-1)} \frac{(q^2\mb;q^2)_n}{(q^2\ma;q^2)_n}  P_n(q^2 \mb z;\mb,\ma,0;q^2).
\eeqa
\end{rem}

\section{A connection between some big $q-$Jacobi polynomials and special little $q-$Jacobi polynomials by tridiagonalization}
Much as in \cite{IK1} and \cite{IK2} we shall  supplement the observations of the preceding section by noting that a large class of big q-Jacobi polynomials can be obtained from a particular tridiagonalization of the special little q-Jacobi operator with (\ref{lab}) for $\ma = q^2$.
\begin{lem}  Consider the little q-Jacobi operator $\tG_1$ and the big q-Jacobi operator $G_1$ of the form:
\beqa
G_1&=&g_1 S_-^2 + g_2 S_-q^{-s_3} +  g_3 S_-q^{s_3} + g_4 q^{2s_3} + g_5q^{-2s_3} + g_6,\\
\tG_1&=& \tg_2 S_-q^{-s_3} +  \tg_3 S_-q^{s_3}  + \tg_4 q^{2s_3} + \tg_5q^{-2s_3} + \tg_6.\label{bigqJ}
\eeqa
For $\tg_3=-q^{2\nu+2}\tg_2$, one has:
\beqa
G_1 \equiv  \beta \tG_1 \bO +  \gamma \tG_1 + \gamma' \bO + \delta \label{TGGexp}
\eeqa
with (\ref{ObO}) and
\beqa
\qquad  \beta&=&  \frac{(q^{\nu+1}g_2 + q^{-\nu-1}g_3)}{(q-q^{-1})(q^{-2\nu-3}\tg_4 - q^{2\nu+3}\tg_5)},\label{sol3}\\
 \gamma &=&   \frac{q^{-\nu-1}(q^{\nu+2}g_3\tg_5 + q^{-\nu-2}\tg_4 g_2)}{\tg_2(q^{-2\nu-3}\tg_4 - q^{2\nu+3}\tg_5)}  ,\quad \gamma'=
-\frac{(q^{-2\nu-2}\tg_4 + q^{2\nu+2}\tg_5 + \tg_6)(q^{\nu+1}g_2 +q^{-\nu-1}g_3  )}{(q-q^{-1})(q^{-2\nu-3}\tg_4 - q^{2\nu+3}\tg_5)},
\nonumber\\
 \delta&=& g_6-\gamma \tg_6.\  \nonumber
\eeqa
\end{lem}
\begin{proof}
Consider the l.h.s. of (\ref{TGGexp}). The monomial $S_-^2$ is easily reduced in terms of $\bO S_-q^{\pm s_3}$, namely:
\beqa
S_-^2 = -\frac{q^{1-\nu}}{q-q^{-1}}\bO S_- q^{-s_3}  +  \frac{q^{\nu-1}}{q-q^{-1}} \bO S_-q^{s_3}.
\eeqa
Consider the r.h.s.  of (\ref{TGGexp}). Using:
\beqa
S_-q^{-s_3}\bO = q^2 \bO S_- q^{-s_3} - \bO^2 q^{\nu+1},\qquad S_-q^{s_3}\bO = q^{-2} \bO S_- q^{s_3} - \bO^2 q^{-\nu-1},
\eeqa
by straightforward calculations, for a slightly more general combination we obtain:
\beqa
\alpha \bO \tG_1 + \beta \tG_1\bO  +\gamma \tG_1 + \gamma' \bO   + \delta &=& 
-  (\alpha + \beta q^{2})(q-q^{-1})q^{\nu-1}\tg_2  S_-^2 \nonumber\\
&& +  \left((\alpha + \beta q^{-2})(q-q^{-1})q^{-\nu}\tg_4 + \gamma \tg_3\right) S_-q^{s_3} \nonumber\\
&& +  \left(-(\alpha + \beta q^{2})(q-q^{-1})q^{\nu}\tg_5 + \gamma \tg_2\right) S_-q^{-s_3} \nonumber\\
&& + \gamma \tg_4 q^{2s_3} + \gamma \tg_5 q^{-2s_3} + \gamma \tg_6 + \delta\nonumber\\
&& + \left( (\alpha + \beta q^{-2})\tg_3 + (\alpha + \beta q^{2})\tg_2 q^{2\nu-2}\right)\bO S_-q^{s_3} \nonumber\\
&& + \left( (\alpha + \beta q^{-2})\tg_4 q^{-2\nu}  +  (\alpha + \beta q^{2})\tg_5 q^{2\nu}+ (\alpha+\beta)\tg_6+\gamma'\right) \bO \nonumber\\
&& - \beta (\tg_2 q^{\nu+1} + \tg_3 q^{-\nu-1} ) \bO^2. \nonumber
\eeqa
Comparing the r.h.s and l.h.s of (\ref{TGGexp}), we identify:
\beqa
&&   g_1 = -(\alpha + \beta q^2)(q-q^{-1}) q^{\nu-1}\tg_2\ ,\label{coefg}\\
&& g_2=-(\alpha + \beta q^{2})(q-q^{-1})q^{\nu}\tg_5 + \gamma \tg_2\ ,\nonumber\\
&&  g_3 = (\alpha + \beta q^{-2})(q-q^{-1})q^{-\nu}\tg_4 + \gamma \tg_3 \ ,\nonumber\\
&& g_4= \gamma \tg_4,\quad  g_5 =  \gamma \tg_5 \ \nonumber, \quad g_6 = \gamma \tg_6 + \delta \nonumber
\eeqa
together with the constraints:
\beqa
&&\tg_2 q^{\nu+1} + \tg_3 q^{-\nu-1}=0 \ ,\\
&&   (\alpha + \beta q^{-2})\tg_3 + (\alpha + \beta q^{2})\tg_2 q^{2\nu-2} =0\ ,\nonumber\\
&& (\alpha + \beta q^{-2})\tg_4 q^{-2\nu}  +  (\alpha + \beta q^{2})\tg_5 q^{2\nu} + (\alpha+\beta)\tg_6+\gamma'=0\ .\nonumber
\eeqa
The first constraint gives $\tg_3=-q^{2\nu+2}\tg_2$. Inserting in the second constraint, one gets $\alpha=0$. The other equations imply
 (\ref{sol3}).
\end{proof}
Note from eq. (\ref{lab}) that the condition $\tg_3 =  - q^{2\nu +2} \tg_2$ implies that $\ma=q^2$. Thus, for a large parameter set, big q-Jacobi polynomials can be constructed from little q- Jacobi polynomials with that special value of the parameter $\ma$ by using the one-sided tridiagonalization introduced in \cite{IK1} and \cite{IK2}.

\section{Conclusion}
This paper has shown how the little and big q-Jacobi polynomials form bases  for representations of the (specialized) Askey-Wilson algebra. The starting point has been the embedding of the Askey-Wilson algebra in $\mathfrak U_q(sl(2))$ realized in terms of q-difference operators. It then proved possible to identify within the model the defining operators of the polynomials; once this was done the algebraic structure could be used to obtain, for instance, the explicit expressions of these polynomials in a simple fashion. The realization was also instrumental in the observation that the little q-Jacobi operator and a tridiagonalization of this operator realized together the Askey-Wilson algebra since they could readily be identified as generators in the equitable embedding of this algebra in $\mathfrak U_q(sl(2))$.\vspace{1mm}

It should similarly prove interesting to adopt the somewhat opposite viewpoint, that is to start from the little and big q-Jacobi operators and their tridiagonalizations and to look at the conditions for these operators to form Askey-Wilson algebras. Recall that the most general tridiagonalization of the ordinary Jacobi operator led to the standard Heun operator \cite{GrVZ}. The program thus suggested would therefore involve certain q-analogs of the Heun equation. An interesting question that is complementary to those pursued in the present paper is: under what specializations will these q-Heun operators form Askey-Wilson algebras with the original operators. We plan to pursue this matter in a future publication.

\vspace{4mm}

\noindent{\bf Acknowledgements:} We thank Paul Terwilliger for comments. L.V. would like to express his gratitude for the hospitality extended to him
by the Institut Denis-Poisson of the Universit\'e Fran\c cois-Rabelais de Tours as Chercheur
Invit\'e where most of this research was carried out. The research of L.V. is funded in part by
a discovery grant from the Natural Sciences and Engineering Research Council (NSERC) of Canada.
 P.B. is supported by C.N.R.S.  Work of A.Z. is supported by the National Science Foundation of China (Grant No.11771015).

\vspace{0.2cm}

\begin{appendix}

\end{appendix}

\end{document}